\documentclass[10pt]{amsart}

\usepackage{amssymb}
\usepackage[shortalphabetic,abbrev]{amsrefs}
\usepackage{enumerate}
\usepackage[all]{xy}
\usepackage[matha,mathx]{mathabx}
\usepackage{graphicx}

\numberwithin{equation}{section}

\theoremstyle{definition}

\theoremstyle{theorem}
\newtheorem{thm}[equation]{Theorem}
\newtheorem*{thm*}{Theorem}

\newtheorem{lemma}[equation]{Lemma}
\newtheorem{prop}[equation]{Proposition}
\newtheorem{coro}[equation]{Corollary}
\newtheorem{ex}[equation]{Example}

\theoremstyle{remark}
\newtheorem{remark}[equation]{Remark}

\newcommand{\exref}[1]{Ex\-am\-ple \ref{#1}}
\newcommand{\thmref}[1]{Theo\-rem \ref{#1}}

\newcommand{\lemref}[1]{Lem\-ma \ref{#1}}
\newcommand{\propref}[1]{Prop\-o\-si\-tion \ref{#1}}
\newcommand{\corref}[1]{Cor\-ol\-lary \ref{#1}}

\newcommand{\remref}[1]{Re\-mark \ref{#1}}

\providecommand{\onto}{\twoheadrightarrow }
\providecommand{\into}{\hookrightarrow }

\DeclareMathOperator{\coker}{coker }
\DeclareMathOperator{\End}{End }
\DeclareMathOperator{\Adj}{Adj }
\renewcommand{\hom}{{\rm Hom}}

\DeclareMathOperator{\GL}{GL}
\DeclareMathOperator{\Isom}{Isom}

\DeclareMathOperator{\im}{im }

\newcommand{\cev}[1]{\reflectbox{\ensuremath{\vec{\reflectbox{\ensuremath{#1}}}}}}

\newcommand{\lversor}{\obackslash}
\newcommand{\rversor}{\oslash}

\newcommand{\lt}[1]{\cev{#1}}
\newcommand{\rt}[1]{\vec{#1}}
\newcommand{\llt}[1]{\overleftarrow{#1}}
\newcommand{\rrt}[1]{\overrightarrow{#1}}

\newcommand{\mor}[1]{\big(\lt{#1},\rt{#1}\big)}

\begin{document}

\title{Division, adjoints, and dualities of bilinear maps}
\author{James B. Wilson}
\address{
	Department of Mathematics\\
	The Ohio State University\\
	Columbus, Ohio 43210\\
}
\email{wilson@math.ohio-state.edu}
\date{\today}

\begin{abstract} 
The distributive property can be studied through bilinear maps and various morphisms between these maps.  The adjoint-morphisms between bilinear maps establish a complete abelian category with projectives and  admits a duality.  Thus the adjoint category is not a module category but nevertheless it is suitably familiar.  The universal properties have geometric perspectives. For example, products are orthogonal sums.  The bilinear division maps are the simple bimaps with respect to nondegenerate adjoint-morphisms.  That formalizes the understanding that the atoms of linear geometries are algebraic objects with no zero-divisors.  Adjoint-isomorphism coincides with principal isotopism; hence, nonassociative division rings can be studied within this framework.

This also corrects an error in an earlier pre-print; see \remref{rem:not-Morita}.
\end{abstract}

\maketitle

\section{Introduction}

We study the distributive property, that is, \emph{bimaps} (also called biadditive or bilinear maps).  As usual, bimaps $B:U\times V\to W$ relate abelian groups $U$, $V$, and $W$ following the relations $(u+u')Bv=uBv+u'Bv$ and $uB(v+v')=uBv+uBv'$, for all $u,u'\in U$ and all $v,v'\in V$ (further definitions in Section \ref{sec:back}).  Distributive products abound in algebra, making this property important to understand.  Yet, this also explains why no one category has emerged for bimaps.   For instance, the multiplication $\cdot:R\times R\to R$ of a ring $R$ is a bimap; but that bimap can be approached as a ring with ring homomorphisms, a left (or right) module probed by linear maps, or a bilinear form\footnote{Here a form means a bimap $B:U\times V\to W$ where $W$ is a cyclic (bi-)module; cf. Section \ref{sec:back}.} where isometries would be the appropriate transformations.  Several less obvious (but still useful) morphisms between bimaps generalize homotopisms from nonassociative algebra \cite{Albert:nonass-I}, derivations from Lie theory \cite{Jacobson:Lie},  crossed-maps from Jordan pairs \cite{Loos:pairs}, etc.   Thus, the class of bimaps admits not just one category, but an ontology (in the Computer Science vocabulary) of multiple categories, functors, and some `nonassociative' categories.  

Here we focus on just three morphism types:   homotopisms,  adjoint-morphisms, and  nondegenerate adjoint-morphisms.  We name categories after their morphisms and not their objects because bimaps admit many incompatible morphisms.  

Albert \cite{Albert:nonass-I} introduced homotopisms for nonassociative rings but his definitions adapt well to all bimaps (Section \ref{sec:homotopism}).  Homotopisms generalize ring homomorphisms, linear mappings, and isometries, and so these are perhaps the most used morphisms between bimaps.  For example, there are several efforts to discover new isotopism classes of semifields (nonassociative rings without zero-divisors) and other \emph{division bimaps} (also called nonsingular bimaps), i.e. $D:U\times V\to W$ with $uDv=0$ implying $u=0$ or $v=0$; cf. \citelist{\cite{Kantor:semifield}*{p. 110},\cite{Shapiro:nonsingular}*{p. 228 \& Chapter 16}}.

Adjoint-morphisms (defined in Section \ref{sec:Adj}) are also common, though usually not considered as a category.  In many ways, adjoint-morphisms are the simplest morphisms for bimaps due to the following: their category is complete, co-complete, abelian, and has projectives (Theorems~\ref{thm:abelian} \& \ref{thm:projectives}).  Adjoint categories also admit a duality -- the transpose.  This implies that despite many similarities, adjoint categories are never equivalent to full module categories (\thmref{thm:adj-not-mod}).  Functorial translations between the homotopism and adjoint categories are essential in applications.  One of the critical relationships is a Galois correspondence set up by tensor products (\thmref{thm:Galois-tensor}).  Though that correspondence appears not to have been formalized until recently (\cite{Wilson:unique-cent}*{p. 2654}), it evolved out of a series of old problems in Group Theory \citelist{\cite{Bayer-Fluckiger:Hasse},\cite{BW:isom},\cite{Wagner:isoms}}. 

The final morphism family involves so-called nondegenerate bimaps, which are perhaps the most common examples of bimaps. In geometric terms, bimaps $B:U\times V\to W$ establish the orthogonality operators
\begin{align}
	X^{\bot} & = \{ v\in V : XBv=0\} & Y^{\top} & = \{u\in U: uBY=0\}
\end{align}
for $X\subseteq U$ and $Y\subseteq V$, and the \emph{radicals} of $B$ are $U^{\bot}$ and $V^{\top}$.  \emph{Nondegenerate} bimaps have trivial radicals.  Notice that the products of unital rings and modules are immediately nondegenerate and bilinear forms are usually assumed to be nondegenerate or a swift argument reduces them to that case.  Adjoint-morphisms are classically interpreted as linear maps that relate the orthogonality operators of one bimap to another  (cf. \lemref{lem:perp-morph}).   So we are concerned with the operators $(\bot,\top)$ as order-reversing mappings between the submodules (subsets) of $U$ and $V$ (indeed they are Galois connections). Here the techniques for forms breakdown substantially because \emph{$(\bot,\top)$ are not bijections}, i.e. they are not \emph{dualities}.  This means that geometric tasks, such as decomposing a bimap into pairwise orthogonal subspaces or determining the isometry group, are no longer trivial observations about standard bases (witness the involved arguments in\citelist{\cite{BW:isom},\cite{Wilson:unique-cent}*{Sections 4--6}}).  So we restrict the lattices (and associated adjoint-morphisms) to the $\bot\top$ and $\top\bot$ stable submodules of $U$ and $V$ respectively, to create a duality.  These restricted \emph{nondegenerate adjoint-morphisms} characterize the minimal intervals in the $\bot\top$ and $\top\bot$ stable sublattices as \emph{nondegenerate adjoint-simple} bimaps (\propref{prop:nondeg-adj}). Here this geometric interpretation translates to a familiar algebraic property.  We show that nondegenerate adjoint-simple bimaps are precisely the division bimaps.  This formalizes the experience that the `atoms' of coordinatized geometry are algebraic objects without zero-divisors (\thmref{thm:division}).

\subsection{Notation}\label{sec:back}

We use $R$ and $S$ for rings, $U$'s for left $R$-modules, and $V$'s for right $S$-modules.   An $(R,S)$-bimodule $W$ is a left $R$-module and a right $S$-module where 
\begin{align*}
	r(ws) & = (rw)s & (\forall r\in R,\forall w\in W,\forall s\in S).
\end{align*}
We write $U_{R^{op}}$ for the right $R^{op}$-module induced from a left $R$-module $U$ and remark that the (co-variant) functor ${}_R U\mapsto U_{R^{op}}$ (which is the identity on homomorphisms) is an isomorphism ${}_R{\tt Mod}\to {\tt Mod}_{R^{op}}$.   We economize on parentheses by the convention that an $R$-linear map $\mu:U\to U'$ of left $R$-modules is evaluated at $u\in U$ by $u\mu$ (so that for all $r\in R$, and all $u\in U$, $ru\mu$ suffices for $r(u\mu)=(ru)\mu$)  and an $S$-linear map $\nu:V\to V'$ of right $S$-modules is expressed by $\nu v$, for each $v\in V$.

Throughout we assume $W={}_R W_S$ is an $(R,S)$-bimodule.  An \emph{$(R,S)$-bimap} is a function $B:U\times V\to W$ of a left $R$-module $U$, a right $S$-module $V$, and an $(R,S)$-bimodule $W$ with the distributive-type properties:
\begin{align}
	(u+u')B v & = uB v+u'B v,\\
	uB(v+v') & = uB v+uB v' & (\forall u,u'\in U,\forall v,v'\in V);
\end{align}
and the associative-type properties:
\begin{align}
	(ru)Bv & = r(uBv), & uB(vs) & = (uBv)s & (\forall u\in U,\forall v\in V,\forall r\in R,\forall s\in S).
\end{align}
\begin{remark}\label{rem:matrices}
We generally think this sort of linear algebra should remain coordinate free; yet, there is some intuition gained by treating $B$ as if it were a matrix.  Indeed, if $U=R^{\oplus m}$ (as row vectors) and $V=S^{\oplus n}$ (as column vectors) then evaluating $B$ on the coordinate bases for $U$ and $V$ specifies an $(m\times n)$-matrix  $M(B)_{ij}=e_i B e_j$ with entries in $W$.  Observe $uBv=uM(B)v$ where the right side is ordinary matrix multiplication.  When $U=V$ we write $uM(B)v^t$.
\end{remark}

We use many basic concepts from modules, categories, and lattices. We refer readers to various other well written sources, mostly  Anderson-Fuller \cite{AF:rings}, Pareigis \cite[Sections 1.7, 1.11, 4.1]{Pareigis}, and \cite{GALOIS}*{Chapter 7}.

\subsection{Homotopism category}\label{sec:homotopism}
For bimaps $B:U\times V\to W$ and $B':U'\times V'\to W'$, a homotopism is a triple $(\phi,\gamma;\kappa)\in \hom_R(U,U')\times \hom_S(V,V')\times \hom_{R,S}(W,W')$,
\begin{align}\label{def:homotopism}
	u\phi C \gamma v & = (uBv)^{\kappa} & (\forall u\in U, \forall v\in V).
\end{align}
As one would expect, these form a category under pointwise composition and this \emph{homotopism} category has kernels (ideals as we shall call them), quotients, and appropriate versions of the isomorphism theorems.  One of the most important ideals of a bimap $B:U\times V\to W$ is its \emph{radical} $\sqrt{B}:V^{\top} \times U^{\bot}\to 0$.   Observe that $B/\sqrt{B}$ is nondegenerate.

Homotopisms admit various relevant subcategories.  For example, if we insist $U=V$ then we may restrict to homotopisms $(\phi,\gamma;\kappa)$ where $\phi=\gamma$.  If further $\phi=\gamma$ and $\kappa$ are invertible then we call $(\phi,\phi;\kappa)$ a \emph{pseudo-isometry}.  If instead we fix $W$ then we often look at so-called \emph{principal} homotopisms $(\phi,\gamma;1_W)$ and, provided also $U=V$, \emph{isometries} are defined as homotopisms of the form $(\phi,\phi;1_W)$.  These also create subcategories.   For \emph{nonassociative algebras}, i.e. bimaps $W\times W\to W$, we obtain the usual \emph{homomorphisms} as homotopisms $(\phi,\phi;\phi)$.

As motivation, we close this brief section with two examples of problems that appear in homotopism categories but are solved by passing to our next category of adjoints.

First, Baer \cite{Baer:class-2} observed that a group $G$, whose center $Z(G)$ contains the commutator $G'$, induces a bimap $B:G/Z(G)\times G/Z(G)\to G'$,
\begin{align}\label{eq:bi-p-groups}
	(xZ(G))B(Z(G)y) & = [x,y] & (\forall x,y\in G).
\end{align}
Isomorphisms between such groups are mapped to pseudo-isometries and in that way the category of pseudo-isometries appears in several works on $p$-groups; e.g. \cite{Higman:enum,Wilson:unique-cent}.  A pressing open problem is to determine ${\rm Aut}~G$ and as a starting point it appears necessary to determine the pseudo-autometries of $B$.  This was recently solved for bimaps $B$ which are also tensor products \cite{BW:aut-tensor} using the Galois connection of Section~\ref{sec:Galois}.

A more subtle use of bimaps arose for intersections of classical subgroups of $\GL(V)$, which had long been studied as algebraic groups with polynomials derived from sets $\Phi=\{\varphi:V\times V\to K_{\varphi}\}$ of bilinear and sesquilinear forms, e.g. \cite{Bayer-Fluckiger:Hasse}.  As exploited in \cite{BW:isom}, this is a modestly disguised problem of a single bimap: set $(\cap\Phi):V\times V\to \bigoplus\{K_{\varphi}: \varphi\in\Phi\}$,
\begin{align}\label{eq:bi-classical}
	(\cap\Phi) (u,v) & = (\varphi(u,v): \varphi\in\Phi)  & (\forall u,v\in V).
\end{align}
Since $\cap_{\varphi\in \Phi} \Isom(\phi)=\Isom(\cap \Phi)$, the structure of this intersection  was immediately extracted from a functorial relationship between the isometry category and the adjoint category; for details see \cite{BW:isom}.  

\section{The category of adjoints}\label{sec:Adj}
The \emph{adjoint} category $\Adj(W)$ has arbitrary $(R,S)$-bimaps $B:U\times V\to W$ (so only $W={}_R W_S$ is fixed) as objects.  Morphisms from $B$ to another $(R,S)$-bimap  $C:U'\times V'\to W$ are pairs $\mor{\mu}\in \hom_R(U,U')\times \hom_S(V',V)$ such that 
\begin{align}\label{def:adjoint}
	u\cev{\mu}C v' & = uB\vec{\mu} v' & (\forall u\in U,\forall v'\in V).
\end{align}
Hereafter we optionally express \eqref{def:adjoint} simply by $\cev{\mu}C=B\vec{\mu}$.   For $(R,S)$-bimaps $B,C$, and $D$ in $\Adj(W)$, we compose $\mor{\mu}\in \Adj(B,C)$ and $\mor{\nu}\in \Adj(C,D)$ by
\begin{align}\label{eq:def-comp}
	\mor{\mu}\mor{\nu} & = (\cev{\mu}\cev{\nu},\vec{\nu}\vec{\mu}).
\end{align}
This establishes a category.  There is an immediate and essential geometric property captured by these morphisms.  For a bimap $B:U\times V\to W$ and subsets $X\subseteq U$ and $Y\subseteq V$, write $X\perp Y$ if $XBY=\{xBy: x\in X,y\in Y\}=0$.

\begin{lemma}\label{lem:perp-morph}
For all $\mor{\mu}\in \Adj(B,C)$,  $(\ker \cev{\mu})\bot( \im \vec{\mu})$ and $(\im \cev{\mu})\bot( \ker \vec{\mu})$.
\end{lemma}
\begin{proof}
For all $u\in \ker\cev{\mu}$, and all $v\in \vec{C}$, $0=u\cev{\mu}Cv=uB\vec{\mu} v$. Also,
$(\lt{B}\cev{\mu})C( \ker \vec{\mu})=(\lt{B})B\vec{\mu}(\ker\vec{\mu})=0$.
\end{proof}

\subsection{A Galois connection between adjoints and tensors}\label{sec:Galois}
Since the adjoint category fixes the codomain of a bimap it may appear that we will not be speaking about tensor products.  To the contrary, one critical aspect of the adjoint category is that it specifies the `best' ring for a tensor product.  That point (which seems to have originated from the study of $p$-groups \citelist{\cite{Wilson:unique-cent}*{p. 2651},\cite{BW:aut-tensor}}) is missed entirely whenever we begin with a diagram of modules over fixed rings.  When the interest begins with bimaps (rather than modules), the notion of a universal ring with which to tensor becomes obvious.

Given $\mor{\mu},\mor{\nu}\in \Adj(B,C)$, the sum $(\cev{\mu}+\cev{\nu},\vec{\mu}+\vec{\nu})\in \Adj(B,C)$ and composition in $\Adj(W)$ distributes over this addition (so that $\Adj(W)$ is a preadditive category).  In particular, if $B:U\times V\to W$ is an $(R,S)$-bimap then the adjoint-endomorphisms form the ring
\begin{equation*}
  \Adj(B) =\{(\cev{\mu},\vec{\mu})\in \End_R U\times \End_S V:  \cev{\mu}B=B\vec{\mu}\}.
\end{equation*}
\begin{remark}\label{rem:adj-generic}
Many historic uses of adjoints focus on a single adjoint pair at a time (even within Category Theory this is the norm).  The uses of an entire adjoint ring are largely restricted to bilinear forms.  Forms are quite limiting as their adjoint rings are simple (when the modules involved have finite chain conditions). By contrast, for every field $k$ and $k$-algebra $A$ there is a $k$-bimap $B:V\times V\to k^3$ with $\Adj(B)\cong A$ \cite{Bayer-Fluckiger:any}. 
\end{remark}

For a ring $A$ we say an $(R,S)$-bimap $B:U \times V\to W$ is \emph{$A$-midlinear} (also called $A$-balanced) if
$U$ is an $(R,A)$-bimodule, $V$ is an $(A,S)$-bimodule, and
\begin{align}
	sB & = Bs & (\forall u\in U,\forall v\in V,\forall s\in A).
\end{align}
There is a universal $A$-midlinear $(R,S)$-bimap for any pair $({}_R U_A, {}_A V_S)$, the \emph{tensor product} $\otimes=\otimes_A:{}_R U_A \times {}_A V_S\to {}_R U\otimes_A V_S$.  Universal there means that every $A$-midlinear $(R,S)$-bimap on $U\times V$ factors through $\otimes$.  However, instead of drawing the usual commutative triangle we draw diagram \eqref{fig:tensor} representing a homotopism in anticipation of a later construction.
\begin{equation}\label{fig:tensor}
\xymatrix{
U_{A} \times {}_{A} V \ar@<-2ex>@{=}[d]\ar@<2ex>@{=}[d] \ar[r]^B & W\\
U_{A} \times {}_{A} V  \ar[r]^{\otimes} & U\otimes V\ar[u]_{\pi(B)}
}
\end{equation}
Though we do not need this generality here, it is a simple matter to extend these notions to $\Omega$-operator groups by insisting instead that $U$ and $V$ are equipped with a function $\phi:\Omega\to \End_R U\times \End_S V$.

Now observe a bimap $B$ is always $\Adj(B)$-midlinear.  Indeed, $\Adj(B)$ is universal with that property in the sense that whenever $B$ is $A$-midlinear then the representation of $A$ in $\End_R U\times \End_S V$ is a subset of $\Adj(B)$.  If $B:U\times V\to W$ and $C:U\times V\to W'$ are two $(R,S)$-bimaps, write $B\preceq C$ if $C$ factors through $B$.  This establishes a partial ordering on the set of bimaps on $U\times V$ and with that we arrive at the following simple but extremely helpful realization: for an $(R,S)$-bimap $B:U\times V\to W$ and representation $\phi:A\to\End_R U\times \End_S V$, 
\begin{equation}
	\otimes_{(A,\phi)} \preceq B\qquad\Longleftrightarrow\qquad 
		A\phi\subseteq \Adj(B).
\end{equation}
That is to say, 
\begin{thm}\label{thm:Galois-tensor}
$\langle U\otimes_{(-)} V, \Adj(-)\rangle$ is a Galois connection (\cite{GALOIS}*{p. 155}) between $(R,S)$-bimaps on $U\times V$ and representations in $\End_R U\times \End_S V$.
\end{thm}

\subsection{Transpose dualities}
Arguably the most important functor for the categories of adjoint-morphisms is the \emph{transpose}, $t=t({}_R W_S)$, that sends $(R,S)$-bimaps $B:U\times V\to W$ to  $(S^{op},R^{op})$-bimaps $B^t:V\times U\to W$,
\begin{align}
	vB^t u & = uBv & (\forall u\in U,\forall v\in V).
\end{align}
Adjoint-morphisms $\mor{\mu}$ are sent to $(\vec{\mu},\cev{\mu})$.  So the transpose is contra-variant.  As $t({}_R W_S)t({}_{S^{op}}W_{R^{op}})=1_{\Adj(W)}$, the pair $\langle t({}_R W_S), t({}_{S^{op}}W_{R^{op}})\rangle$ form a duality between $\Adj({}_R W_S)$ and $\Adj({}_{S^{op}}W_{R^{op}})$.  As a special case, when $k=R=S$ is a commutative ring and $W$ is a (one-sided) $k$-module treated as a $k$-bimodule, the transpose is a contra-variant auto-functor of order $2$.

These dualities have an important consequence.

\begin{thm}\label{thm:adj-not-mod}
$\Adj(W)$ is not is not equivalent to a full category of modules, nor to the dual of a full category of modules.\footnote{Indeed, $\Adj(W)$ cannot be a Grothendieck category since then it would also be a co-Grothendieck; in such categories all objects are zero \cite[p. 116]{Freyd}.}
\end{thm}
\begin{proof}
Module categories are not equivalent to their duals \cite[p. 116]{Freyd}.
\end{proof}

\begin{remark}\label{rem:not-Morita}
Morita initiated the study of $R$-module subcategories that admit dualities \citelist{\cite{AF:rings}*{Section 23},\cite{Morita:dual}}.  In particular, the dualities can be realized by the contra-variant functor $\hom_R(-,W)$ on the \emph{$W$-reflexive} modules, i.e. the modules $U$ for which $U\cong \hom_R(\hom_R(U,W),W)$.  For instance, if $R=k$ is a field and $W=k$ then all finite-dimensional $k$-vector spaces are $k$-reflexive.  However, $\Adj(W)$ is not equivalent to the subcategory of $W$-reflexive $R$-modules, even when $R=k$ is a field.  As mentioned in \remref{rem:adj-generic}, there are $k$-bimaps $B:U\times V\to k^3$ of finite-dimensional $k$-vector spaces whose endomorphism rings, $\Adj(B)$, are arbitrary finite-dimensional $k$-algebras.  By contrast, the endomorphism rings of finite-dimensional $k$-vector spaces are simple $k$-algebras.  A proposed equivalence would be fully faithful and so preserve the endomorphism rings; that is not possible.  

Regrettably, this point was not properly understood in an earlier draft that saw some circulation (arXiv:1007.4329v1); specifically, Proposition 3.62 and Theorem 1.6 were wrong.
\end{remark}

\subsection{Forgetful functors, versors, and adjunction to modules}
With many questions in linear geometry it can help to forget the perpendicularity relations and focus instead on the spaces themselves.  In a formal context this means we introduce the obvious `forgetful' functors from $\Adj(W)=\Adj({}_R W_S)$ to ${}_R {\tt Mod}$ and to ${\tt Mod}_S$.  One of these two functors is covariant and the other contra-variant and the covariant forgetful functor is a right adjoint to a separate functor we call a `versor'.  If we compose the contra-variant forgetful functor with the transpose we get a similar relationship.  These adjunctions to module categories allow us to explore $\Adj(W)$ efficiently.

For each $(R,S)$-bimap $B:U\times V\to W$, and each adjoint-morphism $\mor{\mu}$ in $\Adj(W)$,  define $\lt{B}=U$ and $\llt{\mor{\mu}}=\cev{\mu}$,
and also $\rrt{B}=V$ and $\rrt{\mor{\mu}}=\vec{\mu}$.  Thus, $\Adj(W)$ possess two forgetful functors and the duality in composition requires that exactly one of these be co-variant.  It is a matter of taste when selecting which of these is co-variant: if one prefers left modules then presumably $\llt{(-)}$ should be viewed as co-variant and $\rrt{(-)}$ as contra-variant -- this is our convention here as demonstrated by \eqref{eq:def-comp}.

\begin{lemma}\label{lem:non-deg}
Fix $(R,S)$-bimaps $B$ and $C$ in $\Adj(W)$ and adjoint-morphisms $\mor{\mu}$ and $\mor{\nu}$ from $B$ to $C$. 
\begin{enumerate}[(i)]
\item If $\cev{\mu}=\cev{\nu}$ then $\vec{\mu}\equiv \vec{\nu}$ modulo the right radical of $B$.  So, if $B$ is right nondegenerate ($\lt{B}^{\bot}=0$) then $\vec{\mu}=\vec{\nu}$ and we say that $\cev{\mu}$ determines $\vec{\mu}$. 
\item If $\vec{\mu}=\vec{\nu}$ then $\cev{\mu}\equiv \cev{\nu}$ module the left radical  of $C$; so, if $C$ is left nondegenerate ($\rt{C}^{\top}=0$) then $\cev{\mu}=\cev{\nu}$ and we say that $\vec{\mu}$ determines $\cev{\mu}$.  
\end{enumerate}
\end{lemma}
\begin{proof}  
Let $\cev{\mu}=\cev{\nu}$.  For all $u\in \lt{B}$ and all $v'\in \rt{C}$, 
$uB(\vec{\mu}-\vec{\nu})v'=u\cev{\mu}Cv'-u\cev{\nu}Cv'=0$.
Hence, $(\vec{\mu}-\vec{\nu})\rt{C}\leq \lt{B}^{\bot}$.  The transpose of (i) proves (ii).
\end{proof}

There are also some trivial ways to lift a homomorphism from ${}_R{\tt Mod}$, resp. ${\tt Mod}_S$, to $\Adj(W)$ that will be used at times below.
\begin{lemma}\label{lem:kernels}
Let $B$ be an $(R,S)$-bimap, ${}_R X$ an $R$-module, and $Y_S$ an $S$-module.
\begin{enumerate}[(i)]
\item Let $(\cev{\mu},\vec{\mu})\in \hom_R(X,\lt{B})\times \hom_S(\rt{B},Y)$ with $\vec{\mu}$ is an epimorphism.  Then there is an $(R,S)$-bimap $A$ with $\mor{\mu}\in \Adj(A,B)$ if, and only if, $(\im\cev{\mu})\bot_B (\ker\vec{\mu})$.

\item Let $\mor{\mu}\in \hom_R(\lt{B},X)\times \hom_S(Y,\rt{B})$ and $\cev{\mu}$ an
epimorphism.  Then there is an $(R,S)$-bimap $C$ where $\mor{\mu}\in \Adj(B,C)$ if,
and only if, $(\ker \cev{\mu})\bot_B (\im \vec{\mu})$.
\end{enumerate}
\end{lemma}
\begin{proof}
The forward direction of (i) follows from \lemref{lem:perp-morph}.   Suppose instead that
$\ker\vec{\mu}\leq (\im\cev{\mu})^{\bot}$.  We claim that $A$ is well-defined by the formula $A\vec{\mu}=\cev{\mu}B$.  First, as $\vec{\mu}$ is an epimorphism, for every $x\in X$ and $y\in Y$ there is a $v\in\rt{B}$, $y=\vec{\mu}v$ and so $xAy=xA\vec{\mu}v=x\cev{\mu}Bv$; hence, $A$ is defined on $X\times Y$.  For the well-definedness, take $v'\in \rt{B}$ with $\vec{\mu}v'=y$.  Now $v'-v=z\in \ker\vec{\mu}\leq (\im \cev{\mu})^{\bot}$, therefore,
\begin{align*}
  xAy & = x\cev{\mu}Bv'=x\cev{\mu}Bv+x\cev{\mu}Bz=x\cev{\mu}Bv
\end{align*}
Thus the choice of representative for the pre-image of $y$ is not important.

For (ii) apply the transpose to the result for (i).
\end{proof}

Since we fix the codomain $W$, when we apply the left forgetful functor to a bimap $B:U\times V\to W$, we are forgetting the module $V$ along with the `product' between $U$ and $V$.  So the problem becomes: given $U={}_R U$ and $W={}_R W_S$, recover an $S$-module $U\lversor W$ and an $(R,S)$-bimap $\lversor_R :U\times (U\lversor W)\to W$ through which $B$ is uniquely associated.  The notation $\lversor$ suggests we think of this as a `universal division' of $U$ into $W$:
\begin{align}\label{eq:versor-1}
  	\forall u\in U,\forall w\in W,\forall v\in V, \qquad uBv=w \Rightarrow (v\sigma^B=u\lversor w)\wedge(u\sigma_B=w\rversor v).
\end{align}
There $\sigma^B:V\to U\lversor W$ and $\sigma_B:U\to W\rversor V$ are homomorphisms unique to $B$.  This is analogous to how one may think of tensors, that is, every distributive `product' (bimap) $B:U\times V\to W$  is the image of a `universal product':
\begin{align}
	\forall u\in U,\forall v\in V,\forall w\in W,\qquad uBv=w\Rightarrow (u\otimes v)\pi(B)=w.
\end{align}

Now formally, an $(R,S)$-bimap $\lversor=\lversor_R:U\times (U\lversor_R W)\to W$ is a \emph{left $R$-versor} (respectively a \emph{right $S$-versor} $\rversor=\rversor_S:(W\rversor_S V)\times V\to W$) if for every $(R,S)$-bimap $B:U\times V\to W$, there is a unique homomorphism $\sigma^B:V\to (U\lversor_R W)$ (resp. $\sigma_B:U\to (W\rversor_S V)$) such that corresponding diagram in \eqref{fig:versor} commutes.  These should be compared to our atypical diagram \eqref{fig:tensor} for the tensor.
\begin{equation}\label{fig:versor}
\xymatrix{
{}_R U\times \phantom{a}V \ar@<-3ex>@{=}[d]	\ar@<3ex>[d]^{\sigma^B} \ar[r]^B & W\ar@{=}[d]\\
 {}_R U\times U\lversor W \ar[r]^{\lversor} &  {}_R W
}
\qquad
\xymatrix{
\;\phantom{ab}U\phantom{ab}\times V_S \ar@<-3ex>[d]^{\sigma_B}
		\ar@<4ex>@{=}[d] \ar[r]^B &  W_S\ar@{=}[d]\\
W\rversor V\times V_S \ar[r]^{\rversor} &  W_S
}
\end{equation}
As with tensors, the choice of $U\lversor_R W$ (or $W\rversor_S V$) is unique up to a canonical isomorphism.  The existence of versors is confirmed by taking $U\lversor_R W=\hom_R(U,W)$ (resp. $W\rversor_S V=\hom_S(V,W)$) and defining $u\lversor \tau=u\tau$ (resp. $\tau\rversor v=\tau v$).  In particular, $v\sigma^B=Bv$ (resp. $\sigma_B u=uB$) and is unique because $\lversor$, so defined, is right-nondegenerate (resp. $\rversor$ is left-nondegenerate).  In what follows we sometimes write $U\lversor_R W$ to represent both the bimap as well as the right $S$-module necessary to define the left $R$-versor, much as we do with tensors.

\begin{remark}
The associated \emph{modules} $U\lversor_R W=\hom_S(U,W)$ and $W\rversor_S V=\hom_S(V,W)$ have such established histories that new vocabulary and symbols will be odious to many.  As a modest defense, here we study \emph{bimaps} and care more for the abstraction of `universal division' than any specific construction implementing that property.  Thus, we feel an independent nomenclature is reasonably justified.  For uniformity we harvested a mostly obsolete term `versor' which appears in some of the earliest writings on tensors, e.g. \cite{Ham}*{p. 7}.  We are not aware of any deep connection to the arcane definitions of versors, much like the original intentions for tensors are obscured in the modern treatments.
\end{remark}

As with the tensor, versors induce various functors.  We pick out the functors that lead to bimaps and adjoint-morphisms.  To describe where homomorphisms are sent we prove the following `lifting property'.

\begin{lemma}\label{lem:exists-star}
Fix an $(R,S)$-bimaps $B$.
\begin{enumerate}[(i)]
\item For all $R$-modules ${}_R U$, and all $\alpha\in \hom_R(U,\lt{B})$, there is a unique $\mor{\alpha}\in \Adj(U\lversor_R W,B)$ with $\cev{\alpha}=\alpha$.
\item For all $S$-modules $V_S$, and all $\alpha\in \hom_S(\rt{B},V)$, there is a unique $\mor{\alpha}\in \Adj(B,W\rversor_S V)$ with $\vec{\alpha}=\alpha$.
\end{enumerate}
\end{lemma}
\begin{proof}
For (i), set $\cev{\alpha}=\alpha$ and $C:U\times V'\to W$ to be $C=\cev{\alpha}B$.  Hence there is a unique $\sigma^{C}:V'\to U\lversor_R W$ where for all $u\in U$ and all $v'\in V'$, $u\cev{\alpha}Bv=uCv'=u\lversor (\sigma^{C} v')$.  For (ii), let $D=B\alpha$ and $\mor{\alpha}=(\sigma_D,\alpha)$.
\end{proof}

\begin{remark}
With the suggested choice of $\lversor:U\times \hom_R(U,W)\to W$, $u\lversor\pi= u\pi$, the (left) lifting of $\alpha$ to $\mor{\alpha}$ is simply $\cev{\alpha}=\alpha$ and $\vec{\alpha}=\cev{\alpha}B$.
\end{remark}

For a fixed $(R,S)$-bimodule $W$, $(-)\lversor_R W$ determines a functor ${}_R {\tt Mod}\to\Adj(W)$ sending $U$ to $\lversor:U\times (U\lversor W)\to W$ and assigning to each $\mu\in\hom_R(U,U')$ the unique adjoint-morphism $\mu\lversor W=\mor{\mu}\in \Adj(U\lversor W,U'\lversor W)$ with $\cev{\mu}=\mu$ guaranteed by \lemref{lem:exists-star}(i).  For $\nu\in\hom_R(U',U'')$, $\cev{\mu}\cev{\nu}=\overleftarrow{\mu\nu}$.  As left versors are right-nondegenerate it follows that  $\vec{\nu}\vec{\mu}=\overrightarrow{\mu\nu}$.  Thus, 
\begin{align*}
	(\mu\nu)\lversor W & = (\overleftarrow{\mu\nu},\overrightarrow{\mu\nu}) = (\cev{\mu}\cev{\nu},\vec{\nu}\vec{\mu})
	= (\mu\lversor W)(\nu\lversor W)
\end{align*}
which confirms the only material step in proving that $(-)\lversor_R W$ is a (covariant) functor.   Meanwhile, $W\rversor_S (-)$ is a (contra-variant) functor ${\tt Mod}_S\to \Adj(W)$.  

\begin{thm}\label{thm:adjoint-functor}
For a fixed $(R,S)$-bimodule $W$, the functor $(-)\lversor_R W$ from ${}_R{\tt Mod}$ to $\Adj({}_R W_S)$ is a left adjoint with the left forgetful functor $\llt{(-)}$ its right adjoint. 
\end{thm}
\begin{proof}
Fix a left $R$-module $U$ and an $(R,S)$-bimap $B$ into $W$.  Define $\Phi_{U,B}:\Adj(U\lversor_R W,B)\to \hom_R(U,\lt{B})$ by $\mor{\tau}\Phi_{U,B}=\cev{\tau}$.  By \lemref{lem:exists-star}, $\Phi_{U,B}$ is surjective and by \lemref{lem:non-deg}(i) $\Phi_{U,B}$ is injective.  Now let $C$ be a bimap and $U'$ an $R$-module.  Fix $\mor{\mu}\in \Adj(B,C)$ and $\rho\in \hom_R(U',U)$.  Let $\mor{\rho}=\rho\lversor_R W$ (so $\cev{\rho}=\rho$).  For each $\mor{\tau}\in \Adj(U\lversor_R W,B)$, 
\begin{align*}
	\mor{\tau}\hom(\rho\lversor_R W,\mor{\mu})\cdot \Phi_{U',C} & = 
		\mor{\rho}\mor{\tau}\mor{\mu}\Phi_{U',C} 
		= \cev{\rho}\cev{\tau}\cev{\mu} \\ 
		& = \cev{\tau}\hom(\rho,\cev{\mu}) 
		 = \mor{\tau}\Phi_{U,B}\cdot \hom(\rho,\llt{\mor{\mu}}).
\end{align*}
Thus, $\hom(\rho\lversor_R W,\mor{\mu})\cdot \Phi_{U',C}=\Phi_{U,B}\cdot \hom(\rho,\llt{\mor{\mu}})$ so $\Phi_{U,B}$ is natural in $U$ and $B$.  Hence, $(-)\lversor_R W$ is a left adjoint to $\llt{(-)}$.
\end{proof}

\begin{coro}
The left forgetful functor is continuous (preserves limits).  The right forgetful functor sends limits to colimits.
\end{coro}
\begin{proof} The first part is true of all right-adjoint functors \cite[p. 91]{Pareigis}. The second part follows as $\rrt{(-)}$ is naturally isomorphic to $\llt{(-)^{t}}$ (via the isomorphism of ${\tt Mod}_{S}$ to ${}_{S^{op}}{\tt Mod}$).  Observe also that the transpose sends limits to colimits. 
\end{proof}

\subsection{Adjoints categories are (co-)complete and abelian}\label{sec:prod}
There are several ways to imagine products, coproducts, equalizers, kernels, cokernels, images, etc. for $\Adj(W)$.  To get a head start we use the properties of adjoint functors to provide a hint about the construction on purely formal terms. Sometimes these clues are enough; however, a purely abstract construction will miss the relevance to applications.  So we motivate the construction using the geometric aims.  

First we plan a product for $\Adj(W)$ which mimics the perpendicular sum of 1-spaces used to decompose classical orthogonal and unitary geometries.  The members will be arbitrary modules and not simply $1$-spaces.  Such decompositions appear, for example, in \cite{Wilson:unique-cent}.

Fix  a multiset $\mathcal{B}$ of $(R,S)$-bimaps into $W$. As the left forgetful functor is continuous, it preserves products.  Therefore, a hypothetical product $\obot\mathcal{B}$ for $\mathcal{B}$ in $\Adj(W)$ has $\llt{\obot\mathcal{B}}$ canonically isomorphic to $\prod\cev{\mathcal{B}}$ where $\left\langle \prod\cev{\mathcal{B}}: \{\pi_B:B\in\mathcal{B}\}\right\rangle$ is the usual product in ${}_R {\tt Mod}$. For the right side we work contra-variantly and so we select $\rrt{\obot\mathcal{B}}=\coprod\rt{\mathcal{B}}$ where $\left\langle\coprod\vec{\mathcal{B}},\{\iota_B:B\in\mathcal{B}\}\right\rangle$ is the usual coproduct in ${\tt Mod}_S$.   Combine these constructs into an $(R,S)$-bimap $\obot\mathcal{B}:\prod\cev{\mathcal{B}} \times \coprod\vec{\mathcal{B}} \to W$
following the geometric goal that for every $B\in\mathcal{B}$, $\lt{B}$ must be perpendicular to every member of $\rt{\mathcal{B}}-\{\rt{B}\}$ and $\rt{B}$ should be perpendicular to every member of $\lt{\mathcal{B}}-\{\lt{B}\}$.  This is achieved by the rule:
\begin{align}\label{eq:sum}
		u(\obot\mathcal{B})v
			& = \sum_{B\in\mathcal{B}} u_B B v_B
\end{align}
where $u=(u_B)_{B\in\mathcal{B}}\in\prod\lt{\mathcal{B}}$ and $v=(v_B)_{B\in\mathcal{B}}\in\coprod\rt{\mathcal{B}} $.  (This is a finite sum as we assume $\coprod\rt{\mathcal{B}}$ consist of tuples of finite support.)  
\begin{remark}
If we think of bimaps $B$ and $C$ into $W$ temporarily as matrices with entries in $W$ (cf.  \remref{rem:matrices}), then $B\obot C$ is represented by the matrix $B\oplus C=\left[\begin{smallmatrix} B & 0 \\ 0 & C \end{smallmatrix}\right]$.\footnote{We resist writing this product with $\oplus$ since the product in the category of homotopisms on bimaps has a superseding requirement that $B:U\times V\to W$ and $C:X\times Y\to Z$ have a product $B\oplus C: (U\oplus X)\times (V\oplus Y)\to (W\oplus X)$.}
\end{remark}

Now fix an $(R,S)$-bimap $C$ and a function $\mathcal{B}\to \Adj(C,B)$, $B\mapsto (\lt{\mu}_B, \rt{\mu}_B)$. 
For each $B\in \mathcal{B}$, this give the commutative diagrams in \eqref{eq:prod-1}.
\begin{equation}\label{eq:prod-1}
\xymatrix{\prod\cev{\mathcal{B}}\ar[rr]^{\pi_B} & & \lt{B} &  \textnormal{ and }& \coprod\vec{\mathcal{B}} \ar[dr]_{\vec{\mu}} & & \rt{B}\ar[ll]_{\iota_B}\ar[dl]^{\vec{\mu}_B}\\
 &  \cev{C}\ar[ul]^{\cev{\mu}}\ar[ur]_{\cev{\mu}_B} & & & & \vec{C}. }
\end{equation}
For each $B\in \mathcal{B}$, $(\pi_B,\iota_B)\in \Adj(\obot\mathcal{B}, B)$,
$\mor{\mu}\in \Adj(C,\obot\mathcal{B})$, and $\mor{\mu}(\pi_B,\iota_B)=\mor{\mu_B}$.  
So $\obot$ is a categorical product.   The transpose produces a coproduct $\otop\mathcal{B}=(\obot(\mathcal{B}^t))^t:\coprod\lt{\mathcal{B}}\times \prod\rt{\mathcal{B}}\to W$.  We also recognize that the finite products and coproducts in $\Adj(W)$ are isomorphic. 

\begin{thm}\label{thm:Adj-add}
$\Adj(W)$ is an additive category.
\end{thm}
\begin{proof}
We introduced $\Adj(W)$ as a pre-additive category.  
The zero for $\Adj(W)$ is $0:0\times 0\to W$.  The existence of products and coproducts (that are isomorphic when finite) completes the proof.
\end{proof}
 
We appeal once more to \thmref{thm:adjoint-functor} to give the components of an equalizer and finish off the definition with a geometric observations. 

Because $\Adj(W)$ is additive, a hypothetical equalizer of a pair $\mor{\mu},\mor{\nu}\in \Adj(B,C)$ is the kernel of $\mor{\tau}=\mor{\mu}-\mor{\nu}\in \Adj(B,C)$.  So we actually construct kernels for arbitrary adjoint-morphisms $\mor{\tau}$.  Using continuity of the left forgetful functor, a proposed kernel $\ker\mor{\tau}=\langle D, (\iota,\pi)\rangle$ could take the form $D:\ker\cev{\tau}\times (\rt{C}/\im\vec{\tau})\to W$ and $\iota:\ker\cev{\tau}\to \lt{B}$ and $\pi:\rt{B}\to (\rt{B}/\im\vec{\tau})$ are the associated kernels and cokernels of modules.  Define $D$ by the restriction of $B$:
\begin{align}
	uD(v\pi) & = uD(\im\vec{\tau}+ v)=uBv=u\iota B v & (\forall u\in \ker\cev{\tau},\forall v\in \rt{B}).
\end{align}
However, one should pause to understand that this is well-defined -- a fact that follows instantly from the well used geometric property of adjoints: kernels are perpendicular to images (\lemref{lem:perp-morph}).

Having completed some universal constructions we could stop here and declare that other interesting constructions follow similarly.  A less dubious alternative uses the ingredients here to show $\Adj(W)$ is closed to all limits (products, kernels, pullbacks, images, etc.) and (by the transpose) also to colimits (coproducts, cokernels, pushouts, coimages, etc.).  So we prove:
\begin{thm}\label{thm:abelian}
$\Adj(W)$ is complete and cocomplete abelian category.  
\end{thm}

\begin{proof}[Proof of \thmref{thm:abelian}]
Since $\Adj(W)$ has arbitrary (co-)products and (co-)equalizers it is a (co-)complete category \cite[p. 85]{Pareigis}.  As $\Adj(W)$ is additive (\thmref{thm:Adj-add}) and has kernels and cokernels, it remains only to show that every monomorphism $\mor{\mu}\in\Adj(B,C)$ is kernel and every epimorphism is a cokernel.  From the construction above we notice a monomorphism is the kernel of its cokernel and the transpose completes the proof.
\end{proof}

\subsection{Projective bimaps}
We end our brief list of properties for $\Adj(W)$ by describing projectives.  In particular it is possible to define (co-)homology for bimaps.  Thus far the constructions in $\Adj(W)$ took the sensible route of pairing up a well-known construction in the left with its dual in the right.  This is impossible for projectives as demonstrated by the next example.
\begin{ex}\label{ex:proj}
If $P$ is a finitely generated projective abelian group, $Q$ is a finitely generated injective abelian group, $W$ is a torsion group, and $B:P\times Q\to W$ is a $\mathbb{Z}$-bimap, then $B$ is everywhere zero.
\end{ex}
\begin{proof}
As $P$ and $Q$ are finitely generated, $P$ is free and $Q$ is divisible.  Hence, $B$ factors through $P\times Q\to (P\otimes Q)\cong Q^m$.  The homomorphism $\pi(B):Q^m\to W$ must be trivial as $Q^m$ is divisible and $W$ is a torsion module.
\end{proof}

Since projectives are not limits (or colimits), the continuity of the forgetful functors is not applicable as it was above.
In this way, \exref{ex:proj} does not prevent the development of projectives, though it does change the approach.  Ultimately projective bimaps in $\Adj(W)$ are obtained through versors, but that construction has the benefit of hindsight.  Indeed, the abstraction of versors grew out of the characterization of projective bimaps as evident in \thmref{thm:projectives}.  To prove that theorem we first characterize monomorphisms and epimorphisms in $\Adj(W)$.

\begin{lemma}\label{lem:epic-monic}
Let $\mor{\mu}$ be an adjoint-morphism.  The following hold.
\begin{enumerate}[(i)]
\item $\mor{\mu}$ is monic if, and only if, $\cev{\mu}$ is monic and $\vec{\mu}$ is epic.
\item $\mor{\mu}$ is epic if, and only if, $\cev{\mu}$ is epic and $\vec{\mu}$ is monic.
\item $\mor{\mu}$ is an adjoint-isomorphism if, and only if, 
$\cev{\mu}$ and $\vec{\mu}$ are isomorphisms.
\end{enumerate}
\end{lemma}
\begin{proof}
Consider (i) in the forward direction.  Let $\mor{\mu}\in \Adj(B,C)$ be monic.  

Let $U$ be an arbitrary left $R$-module and $\nu,\pi\in \hom_R(U,\lt{B})$ such that $\nu\cev{\mu}=\pi\cev{\mu}$.  By \lemref{lem:exists-star}, there are $\mor{\nu},\mor{\pi}\in \Adj(U\lversor_R W,B)$ with $\nu=\cev{\nu}$ and $\pi=\cev{\pi}$. Since $\cev{\nu}\cev{\mu}=\cev{\pi}\cev{\mu}$ and $U\lversor_R W$ is right nondegenerate, by \lemref{lem:non-deg}(i), $(\cev{\nu}\cev{\mu},\vec{\mu}\vec{\nu})=(\cev{\pi}\cev{\mu},\vec{\mu}\vec{\pi})$ and so $\mor{\nu}\mor{\mu}=\mor{\pi}\mor{\mu}$.  By assumption $\mor{\mu}$ is monic so $\mor{\nu}=\mor{\pi}$ and so $\cev{\nu}=\cev{\pi}$.  Therefore, $\nu\cev{\mu}=\pi\cev{\mu}$ implies $\nu=\pi$ and so $\cev{\mu}$ is monic. 

To show $\vec{\mu}$ is epic, suppose $V$ is a right $S$-module and $\vec{\nu},\vec{\pi}\in\hom_R(\rt{B},V)$ such that $\vec{\mu}\vec{\nu}=\vec{\mu}\vec{\pi}$.  Let $A$ be the trivial bimap $0\times V\to W$.  For all $y\in \rt{B}$, $0By=0=0A\vec{\nu}y=0A\vec{\pi}y$; so, $(0,\vec{\nu}),(0,\vec{\pi})\in \Adj(A,B)$.  Now $(0,\vec{\nu})\mor{\mu}=(0,\vec{\mu}\vec{\nu})=(0,\vec{\mu}\vec{\pi})=(0,\vec{\pi})\mor{\mu}$. As $\mor{\mu}$ is monic, $(0,\vec{\nu})=(0,\vec{\pi})$; so, $\vec{\nu}=\vec{\pi}$ and $\vec{\mu}$ is epic.

For the converse, we know $\cev{\mu}$ and $\vec{\mu}$ have dual cancellation properties and so $\mor{\mu}$ has the cancellation property of $\cev{\mu}$, i.e. $\mor{\mu}$ is monic.  This prove (i).

Part (ii) follows from the transpose applied to (i) and (i) and (ii) imply (iii).
\end{proof}

We encountered the categorical product in $\Adj(W)$ as a bimap which is a module product on the left and a module coproduct on the right.  However, bimaps offer the flexibility of using universal constructions on just one side.  The \emph{left semi-product} of bimaps $B$ and $C$ where  $\lt{B}=\lt{C}$ is $[B,C]:\lt{B}\times (\rt{B}\oplus \rt{C})\to W$ with
\begin{align}
	u[B,C](v\oplus v') & = uBv+uCv' & (\forall u\in \lt{B}, \forall v\in \rt{B},\forall v'\in \rt{C}).
\end{align}
Intuitively we think of $[B,C]$ as a partitioned matrix.  The meaning of $\begin{bmatrix}B\\ C\end{bmatrix}$ for bimaps $B$ and $C$ with $\rt{B}=\rt{C}$ is similarly understood. 

\begin{thm}\label{thm:projectives}
If $P$ is a projective in $\Adj(W)$, then $\lt{P}$ is a projective $R$-module and $\rt{P}\to\hom_R(\lt{P},W)$ given by $y\mapsto Py$ is an epimorphism.
\end{thm}
\begin{proof}
Let $\lt{\mu}\in \hom_R(U,U')$ be an epimorphism and $\lt{\nu}\in \hom_R(\lt{P},U')$.  Take $B:U\times 0\to W$ and $C:U'\times 0\to W$ (which are both 0 everywhere) and notice $(\lt{\mu},0)\in \Adj(B,C)$ is an epimorphism and $(\lt{\nu},0)\in \Adj(P,C)$.  As $P$ is projective there is a $\mor{\tau}\in \Adj(P,B)$ such that $\mor{\tau}(\lt{\mu},0)=(\lt{\nu},0)$.  Thus, $\lt{\tau}\lt{\mu}=\lt{\nu}$ and so $\lt{P}$ is projective in ${}_R {\tt Mod}$.\footnote{Note that construction cannot be modified to show $\rt{P}$ is injective in ${\tt Mod}_S$.}  Now concentrate on the second claim.

We treat $\pi\in \hom_R(\lt{P},W)$ as an $(R,S)$-bimap $\lt{P}\times S\to W$ where $(x, s)\mapsto x\pi s$.  It follows that $\mor{\mu}=(x\mapsto x,y\mapsto y\oplus 0)$ is an epimorphism $[P,\pi]\to P$.  As $P$ is projective, the epimorphism $\mor{\mu}$ splits and so there exists a monomorphism $\mor{\nu}\in \Adj(P,[P,\pi])$ such that $\mor{\nu}\mor{\mu}=1_P$.  As $\cev{\mu}=1$, so does $\cev{\nu}$.  Also, $\mor{\mu}\mor{\nu}=\mor{\epsilon}=(1,\vec{\epsilon})$ is an idempotent endomorphism of $[P,\pi]$, and so is $(1-\cev{\epsilon},1-\vec{\epsilon})=(0,1-\vec{\epsilon})$.  As $\ker (1-\vec{\epsilon})=\rt{P}\oplus 0$ and $\rt{P}\oplus S=\ker (1-\vec{\epsilon})\oplus \im (1-\vec{\epsilon})$, there is a $y_{\pi}\in \rt{P}$ such that $y_{\pi}\oplus (-1)\in \im (1-\vec{\epsilon})$.  Also, $\ker (1-\cev{\epsilon})=\lt{P}$ is $[P,\pi]$-perpendicular to $\im (1-\vec{\epsilon})$ and so
\begin{align*}
	0 & = x [P,\pi] (y_{\pi}\oplus (-1) ) = xPy_{\pi}-x\pi & (\forall x\in \lt{P}).
\end{align*}
Therefore $xPy_{\pi}=x\pi$, for all $x\in \lt{P}$.  As $\pi$ is arbitrary this says that for each
$\pi\in \hom_R(\lt{P},W)$, there is a $y_{\pi}\in \rt{P}$ such that $Py_{\pi}=\pi$.
\end{proof}

\begin{coro}\label{coro:projectives}
If $P$ is projective in $\Adj(W)$ then there is a unique monomorphism $\mor{\iota}:\lt{P}\lversor W\hookrightarrow P$ with $\cev{\iota}$ the identity.
\end{coro}

Following \corref{coro:projectives} we are satisfied to describe projectives as essentially versors.  This is the final point to verify.

\begin{thm}\label{thm:projective}
If $P$ is projective in ${{}_R{\tt Mod}}$ then $P\lversor_R W$ is projective in $\Adj(W)$.
\end{thm}
\begin{proof}
Let $\mor{\mu}\in \Adj(B,C)$ be an epimorphism and  $\mor{\nu}\in\Adj(P\lversor W,C)$.  As $P$ is projective and $\cev{\mu}$ is an epimorphism, there is $\cev{\tau}\in\hom_R(P,\rt{B})$ such that $\cev{\tau}\cev{\mu}=\cev{\nu}$.  By \lemref{lem:exists-star} there is a unique $\mor{\tau}\in \Adj(P\lversor W,B)$ extending $\cev{\tau}$.  Now
$\mor{\tau}\mor{\mu}=(\cev{\tau}\cev{\mu},\vec{\mu}\vec{\tau})=(\cev{\nu},\vec{\mu}\vec{\tau})$.  As $P\lversor W$ is right nondegenerate, $\vec{\mu}\vec{\tau}=\vec{\nu}$ (\lemref{lem:non-deg}(i)).  Thus,
$\mor{\tau}\mor{\mu}=\mor{\nu}$ proving that $P\lversor W$ is projective.
\end{proof}

\begin{remark}
An abelian category is said to have \emph{enough} projectives if every object is the epimorphic image of a projective.  As module categories have enough projectives, for every bimap $B$ into $W$ there is a projective $P$ and an epimorphism $\lt{\mu}:P\to \lt{B}$ which extends uniquely to $\mor{\mu}\in \Adj(P\lversor_R W,B)$.  However, $\mor{\mu}$ need not be epic as $\rt{\mu}$ might not be monic.  For example, if $B:0\times V\to W$ and $V\neq 0$, then $\Adj(P\lversor_R W, B)=0$ for all $R$-modules $P$ (so no element is an epimorphism).  
\end{remark}

\section{Nondegenerate bimaps}\label{sec:non-deg}\label{sec:nondeg}
In this section we focus on nondegenerate bimaps and nondegenerate adjoint-morphisms.  These do not form a subcategory of $\Adj(W)$ (under any choice of $W\neq 0$) but are nevertheless quite robust.  For example, products, kernels, images, and their duals remain nondegenerate when we start with nondegenerate bimaps and nondegenerate adjoint-morphisms.     Unexpectedly, division maps are precisely the simple objects with respect to nondegenerate adjoint-morphisms (\thmref{thm:division}).  

Recall form Section \ref{sec:homotopism} that for every bimap $B$ in $\Adj(W)$, we may treat the radical of $B$ as an ideal $\sqrt{B}:V^{\top}\times U^{\bot}\to 0$ in the homotopism category and so we can pass to a nondegenerate bimap $B^{\surd}=(B/\sqrt{B}):U/V^{\top} \times V/U^{\bot}\to W$ where
\begin{align*}
		(u+V^{\top})(B^{\surd})(U^{\bot}+v) & = uBv, & (\forall u\in U, \forall v\in V).
\end{align*}
Furthermore, for $\mor{\mu}\in \Adj(B,C)$, define
\begin{align*}
	\mor{\mu}^{\surd} & = (\cev{\mu}/\sqrt{\cev{\mu}}, \vec{\mu}/\sqrt{\vec{\mu}})\in \hom_R(\lt{B}/\rt{B}^{\top}, \lt{C}/\rt{C}^{\top})\times \hom_S(\rt{C}/\lt{C}^{\bot},\rt{B}/\lt{B}^{\bot}).
\end{align*}
The full subcategory $\Adj(W)$ of nondegenerate bimaps is a reflexive additive subcategory, but it is not an abelian subcategory as the following example demonstrates.

\begin{ex}\label{ex:deg-morph}
If $K$ is a proper field extension of a field $k$, then set $C:K\times k\to K$ to be
$xCa=xa$, for all $x\in K$ and all $a\in k$. Also define $B:K\times K\to K$ so that 
$xBy=xy$, for all $x,y\in K$.  There is an adjoint-epimorphism
$(1,\iota:k\hookrightarrow K)\in \Adj(B,C)$ whose kernel is the restriction of $B$ to the degenerate bimap
$0\times K/k\to K$ (where here $K/k$ denotes the quotient as additive groups).  
\end{ex}

\subsection{Nondegenerate adjoint-morphisms}
Kernels of adjoint morphisms between two nondegenerate bimaps can be degenerate (\exref{ex:deg-morph}).  Hence, we restrict adjoint-morphisms to those whose kernels and cokernels are nondegenerate as well.  We call these \emph{nondegenerate} adjoint-morphisms.

\begin{lemma}\label{lem:perp-adj-tool}
Let $\mor{\mu}\in \Adj(B,C)$.
\begin{align*}
	X^{\bot}\cap \im \vec{\mu} & = \vec{\mu}((X\cev{\mu})^{\bot}) & (\forall X\subseteq \lt{B}),\\
	Y^{\top}\cap\im \cev{\mu} & = (\vec{\mu}Y)^{\top} \cev{\mu} & (\forall Y\subseteq \rt{C}).
\end{align*}
Indeed, if $B$ is right-nondegenerate and $C$ is left-nondegenerate then $(\im \cev{\mu})^{\bot}=\ker \vec{\mu}$
and $(\im \vec{\mu})^{\top}=\ker \cev{\mu}$.
\end{lemma}
\begin{proof}
First, $\vec{\mu}((X\cev{\mu})^{\bot}) = \vec{\mu}\{v'\in \rt{C}: 0=(X\cev{\mu}) Cv'\}=\{\rt{\mu}v': v'\in \rt{C}, 0=XB\vec{\mu}v'\}=X^{\bot}\cap \im \vec{\mu}$.   If $B$ is right-nondegenerate, then $0=\lt{B}^{\bot}\cap \im\rt{\mu}=\rt{\mu}((\im\lt{\mu})^{\bot})$; thus, $(\im \lt{\mu})^{\bot}\leq \ker\rt{\mu}$.  By \lemref{lem:perp-morph}, $\ker\rt{\mu}\leq \im (\lt{\mu})^{\bot}$.
The rest follows similarly.
\end{proof}

\begin{prop}\label{prop:nondeg-adj}
Let $B$ and $C$ be nondegenerate bimaps in $\Adj(W)$.
For $\mor{\mu}\in \Adj(B,C)$ the following are equivalent.
\begin{enumerate}[(i)]
\item $(\ker \cev{\mu})^{\bot}=\im \vec{\mu}$, and $(\ker \vec{\mu})^{\top}=\im \cev{\mu}$.

\item $(\ker \cev{\mu})^{\bot}\leq\im \vec{\mu}$, and $(\ker \vec{\mu})^{\top}\leq\im \cev{\mu}$.

\item $\ker \cev{\mu}=(\ker \cev{\mu})^{\bot\top}$, $\im \cev{\mu}=(\im\cev{\mu})^{\bot\top}$,
$\ker \vec{\mu}=(\ker\vec{\mu})^{\top\bot}$, and $\im \vec{\mu}=(\im\vec{\mu})^{\top\bot}$.

\end{enumerate}
\end{prop}
\begin{proof}
Evidently (i) implies (ii).  By \lemref{lem:perp-morph}, (ii) implies (i).

Suppose (i).  By \lemref{lem:perp-adj-tool}, $(\im \cev{\mu})^{\bot\top} = ( \ker \vec{\mu})^{\top}=\im \cev{\mu}$; thus also $(\ker\vec{\mu})^{\top\bot}=(\im \cev{\mu})^{\bot}=\ker \vec{\mu}$.  The rest follow similarly so that (i) implies (iii).

Suppose (iii).  By \lemref{lem:perp-adj-tool}, $(\ker \vec{\mu})^{\top}=(\im \cev{\mu})^{\bot\top}=\im \cev{\mu}$.
Similarly, $(\ker \cev{\mu})^{\bot}=\im \vec{\mu}$.  Hence, (iii) implies (i).
\end{proof}

We say $\mor{\mu}$ is \emph{nondegenerate} if it satisfies any of the properties in \propref{prop:nondeg-adj}.
The following short-cuts can also help in proving nondegeneracy of adjoint-morphisms.

\begin{prop}\label{prop:iso-nondeg}
For  bimaps $B$ and $C$, $\mor{\mu}\in \Adj(B,C)$ is nondegenerate if:
\begin{enumerate}[(i)]
\item $\mor{\mu}$ is monic, $B$ is right-nondegenerate, $C$ is nondegenerate, and $(\ker \vec{\mu})^{\top}\leq \im \cev{\mu}$.

\item $\mor{\mu}$ is epic, $B$ is nondegenerate, $C$ is left-nondegenerate, and $(\ker \cev{\mu})^{\bot}\leq \im \vec{\mu}$.

\item $\mor{\mu}$ is an isomorphism, $B$ is right-nondegenerate and $C$ is left-nondegenerate.
\end{enumerate}
\end{prop}
\begin{proof}
For (i), as $\mor{\mu}$ is monic, $\ker \cev{\mu}=0$, $\im \vec{\mu}=\rt{B}$, and $(\ker \cev{\mu})^{\bot}=\im \vec{\mu}$.  By \lemref{lem:non-deg}(iii),  $\rt{B}^{\top}\cev{\mu}\leq \rt{C}^{\top}$.  As $\cev{\mu}$ is monic and $C$ is nondegenerate, it follows that $(\rt{B})^{\top}=0$ so that $B$ is left-nondegenerate.  As $B$ is also right-nondegenerate, $B$ is nondegenerate.  By \propref{prop:nondeg-adj}(ii), $\mor{\mu}$ is nondegenerate.

Part (ii) follows the transpose of (i) and part (iii) follows from (i) and (ii).
\end{proof}

Once we have a nondegenerate adjoint-morphism we can focus on a single component of the pair, for example.
\begin{prop}\label{prop:one-variable}
If $\mor{\mu}\in \Adj(B,C)$ is nondegenerate then $\cev{\mu}$ is a monomorphism, epimorphism, or isomorphism, if, and only if, $\vec{\mu}$ is an epimorphism, monomorphism, or isomorphism, respectively.
\end{prop}
\begin{proof}
If $\cev{\mu}$ is monic then $\rt{B}=0^{\bot}=(\ker \cev{\mu})^{\bot}=\im \vec{\mu}$; so $\vec{\mu}$ 
is epic.  Next, if $\cev{\mu}$ is epic then $\lt{C}=\im \cev{\mu}=(\ker \vec{\mu})^{\top}$.  As $C$ is nondegenerate,
this forces $\ker \vec{\mu}=0$.  The rest follows by application of the transpose.
\end{proof}

Following \lemref{lem:perp-morph}, we understood that the orthogonality relations induced by a bimap $B:U\times V\to W$ can be studied through adjoint-morphisms.  The operators $(\bot,\top)$ form a Galois connection (here our definitions use the order-reversing form)
from the lattice of subsets (submodules) of $U$ to the lattice of subsets (submodules) of $V$; cf. \cite{GALOIS}*{p. 155}.  We will need the following observation.

\begin{thm}\label{thm:perp-lattice}
Let $B$ be a bimap.  The map $X\mapsto X^{\bot}$ is a complete lattice anti-isomorphism from the $\top\bot$-stable submodules of $U$ to the $\bot\top$-stable submodules of $V$ (the inverse is $Y\mapsto Y^{\top}$).  In particular, for all sets $\mathcal{X}$ of submodules of $U$,
\begin{align*}
	\left(\bigcap\mathcal{X}\right)^{\bot} & = \left(\sum\mathcal{X}^{\bot}\right)^{\top\bot}.
\end{align*}
\end{thm}
\begin{proof} This follows from interpreting \cite{GALOIS}*{Proposition 7.31} in our context.\end{proof}

\subsection{The ersatz-category of nondegenerate bimaps}
First we show that the composition of nondegenerate adjoint-morphisms can be degenerate, \exref{ex:comp-fail}.  Hence, we describe the results within an \emph{ersatz-category} $\Adj^{\surd}(W)$, that is, the collection of nondegenerate bimaps and nondegenerate adjoint-morphisms, which despite not begin closed to all compositions behaves as a category wherever composition is defined.  Hence we limp along with weaker properties that apply in critical cases.

\begin{ex}\label{ex:comp-fail}
Let $k$ be a field, $V=k^n$ for some integer $n\geq 3$, and $\{x,y\}\subseteq V$ a $k$-linearly independent subset of size 2.  Define $W=V\wedge V$, $A:kx\times V/kx\to W$, 
$B:V\times V\to W$, and $C:V/ky\times ky\to W$ by
\begin{align*}
	uA(kx+v) & = u\wedge v & (\forall u\in kx,\forall v\in V),\\
	uBv & = u\wedge v & (\forall u,v\in V),\textnormal{ and }\\
	(u+ky)Cv & = u\wedge v & (\forall u\in V,\forall v\in ky).
\end{align*}
It follows that $(\iota_A:kx\into V,\pi_A:V\onto V/kx)\in \Adj(A,B)$ is a nondegenerate monomorphism and $(\pi_C:V\onto V/ky,\iota_C:ky\into V)\in \Adj(B,C)$ is a nondegenerate epimorphism.  However, $(\iota_A,\pi_A)(\pi_C,\iota_C)$ is degenerate.
\end{ex}

\begin{remark}
We caution that it is often possible that a universal construction leads to a nondegenerate bimap but where the associated adjoint-morphisms are degenerate.  So when we speak of a nondegenerate universal property we mean for \emph{both} the bimaps and the adjoint-morphisms to be nondegenerate. 
\end{remark}

In light of \exref{ex:comp-fail} it is critical to show that a some important composites of nondegenerate adjoint-morphisms are nondegenerate.  The following technical rule will be useful in that effort.

\begin{thm}\label{thm:nondeg-comp}
Let $\mor{\mu}\in \Adj(A,B)$ and $\mor{\nu}\in \Adj(B,C)$ be nondegenerate.
It follows that $\mor{\mu}\mor{\nu}$ is nondegenerate if, and only if,
\begin{align*}
	(\im \cev{\mu}+\ker \cev{\nu})^{\bot\top} \cev{\nu} & \leq \im \cev{\mu}\cev{\nu}, & 
	\vec{\mu}(\im \vec{\nu}+\ker \vec{\mu})^{\top\bot}  & \leq \im \vec{\nu}\vec{\mu}.
\end{align*}
Indeed equality can also be used in place of inequalities.
\end{thm}
\begin{proof}
As $\ker \cev{\mu}\leq \ker \cev{\mu}\cev{\nu}$, $(\ker \cev{\mu}\cev{\nu})^{\bot}\leq (\ker \cev{\mu})^{\bot}=\im \vec{\mu}$.  By \lemref{lem:perp-adj-tool}
and \thmref{thm:perp-lattice},
\begin{align*}
	(\ker \cev{\mu}\cev{\nu})^{\bot} & = (\ker \cev{\mu}\cev{\nu})^{\bot}\cap \im \vec{\mu} 
		 = \vec{\mu}((\ker \cev{\mu}\cev{\nu})\cev{\mu})^{\bot}\\
		& = \vec{\mu}((\ker \cev{\nu}\cap \im \cev{\mu})^{\bot})
		 = \vec{\mu}(((\ker \cev{\nu})^{\bot} + (\im \cev{\mu})^{\bot})^{\top\bot})\\
		& = \vec{\mu}((\im \rt{\nu}+\ker \rt{\mu})^{\top\bot}).
\end{align*}
By the transpose, $(\ker \vec{\nu}\vec{\mu})^{\top}=(\im \cev{\mu}+\ker \cev{\nu})^{\bot\top}\cev{\nu}$.
By \propref{prop:nondeg-adj}, $\mor{\mu}\mor{\nu}$ is nondegenerate if, and only if,
\begin{align*}
	\im \cev{\mu}\cev{\nu} & \geq  (\ker \vec{\nu} \vec{\mu})^{\top}=(\im \cev{\mu}+\ker \cev{\nu})^{\bot\top} \cev{\nu}; \textnormal{ and} \\
	\im \vec{\nu}\vec{\mu} & \geq  (\ker \cev{\mu}\cev{\nu})^{\bot}  = \vec{\mu}((\im \vec{\nu}+\ker \vec{\mu})^{\top\bot}).
\end{align*}
\end{proof}

\begin{coro}\label{coro:comp-rule}
Let $\mor{\mu}\in \Adj(A,B)$ and $\mor{\nu}\in \Adj(B,C)$ be nondegenerate.
If $\mor{\mu}$ is epic or $\mor{\nu}$ is monic then $\mor{\mu}\mor{\nu}$ is nondegenerate.
\end{coro}
\begin{proof}
If $\mor{\mu}$ is epic then $\im \cev{\mu}=\lt{B}$ and $\ker\vec{\mu}=0$; hence,
\begin{align*}
	(\im \cev{\mu}+\ker \cev{\nu})^{\bot\top}\cev{\nu}  & = (\im \cev{\mu})^{\bot\top} \cev{\nu} = \im \cev{\mu}\cev{\nu};\\
	\vec{\mu}(\im \vec{\nu}+\ker \vec{\nu})^{\top\bot}  & = \vec{\mu}(\im \vec{\nu}) = \im \vec{\nu}\vec{\mu}.
\end{align*}
By \thmref{thm:nondeg-comp}, $\mor{\mu}\mor{\nu}$ is nondegenerate. \end{proof}

We now have the tools required to show $\Adj^{\surd}(W)$ is closed to many of the familiar universal constructions such as kernels, images, and their duals.  Though products are not always nondegenerate (their associated adjoint-morphism need not be nondegenerate) often these are nondegenerate.

\begin{prop}\label{prop:nondeg-morph}
The kernel and cokernel of a nondegenerate adjoint-morphism are nondegenerate.
\end{prop}
\begin{proof}
Let $\mor{\mu}\in \Adj(B,C)$ be nondegenerate.  Take $\coker\rt{\mu}=\rt{B}/\im \rt{\mu}$ and $A=\ker \mor{\mu}:\ker \lt{\mu}\times \coker \rt{\mu} \to W$.  Now
\begin{align*}
	(\ker \lt{\mu})^{\bot} & =\{\im \rt{\mu}+v\in \coker\rt{\mu}: (\ker \rt{\mu})A(\im \rt{\mu}+v)=0\}\\
		& = \{\im \rt{\mu}+v\in \coker \rt{\mu}: (\ker \lt{\mu})Bv=0\}\\
		& = (\ker \lt{\mu})^{\bot}/\im \rt{\mu}=0,
\end{align*}
where the last equality follows by \propref{prop:nondeg-adj}.  
Hence, $A$ is right-nondegenerate.  Furthermore, $(\iota:\ker \lt{\mu}\into \lt{B},\pi: \rt{B}\onto \coker \rt{\mu})\in \Adj(A,B)$
and is monic.  By \propref{prop:nondeg-adj}, $(\ker \pi)^{\top}  =(\im \rt{\mu})^{\top}=(\ker \lt{\mu})^{\bot\top}=\ker\lt{\mu}=\im \iota$.
So by \propref{prop:iso-nondeg}(i), $(\iota,\pi)$ is nondegenerate; thus, kernels (and their implied inclusion maps) are nondegenerate.  
Cokernels follow by application of the transpose.
\end{proof}

\begin{prop}
Let $C$ be a nondegenerate bimap, $\mathcal{B}$ be a multiset of nondegenerate bimaps into $W$, and $\{\mor{\mu}_B\in \Adj(C,B): B\in \mathcal{B}\}$ a $\mathcal{B}$-indexed set of nondegenerate adjoint-morphisms.  Take $\langle \obot \mathcal{B},\{\mor{\pi}_B:B\in\mathcal{B}\}\rangle$ as the product in $\Adj(W)$ and $\mor{\mu}\in \Adj(C,\obot \mathcal{B})$ with, for all $B\in \mathcal{B}$, $\mor{\mu}\mor{\pi}_B=\mor{\mu}_B$.
\begin{enumerate}[(i)]
\item $\obot \mathcal{B}$ is nondegenerate and, for all $B\in\mathcal{B}$, $\mor{\pi}_B$ is nondegenerate.
\item $\mor{\mu}$ is nondegenerate if, and only if, $\im \rt{\mu}$ (equivalently 
$\sum\{\im \rt{\mu}_B:B\in\mathcal{B}\}$) is $\top\bot$-stable in $\rt{C}$.
\item If there is a $B\in \mathcal{B}$ such that $\mor{\mu}_B$ is monic, then $\mor{\mu}$ is nondegenerate.
\end{enumerate}
\end{prop}
\begin{proof}
For (i), observe that $\llt{(\obot \mathcal{B})}^{\bot}=\prod\{\lt{B}^{\bot}:B\in\mathcal{B}\}=0$ and similarly $\obot \mathcal{B}$ is also left-nondegenerate.  For each $B\in\mathcal{B}$, put $\mor{\pi}_B=(\pi_B,\iota_B)$ so that
\begin{align*}
	(\ker \pi_B)^{\bot} & = \left(\prod (\lt{\mathcal{B}}-\{\lt{B}\})\right)^{\bot}
		 = \rt{B}\iota_B.
\end{align*}
As $B$ is nondegenerate and $(\pi_B,\iota_B)$ is epic, by \propref{prop:iso-nondeg}(ii), $\mor{\pi}_B$ is nondegenerate.

Next we show (ii).  For each $B\in\mathcal{B}$, put $(\lt{\mu}_B,\rt{\mu}_B)=\mor{\mu}_B$.  Now $(\ker \rt{\mu}_B)\iota_B \rt{\mu}=(\ker \rt{\mu}_B)\rt{\mu}_B=0$ so that $(\ker \rt{\mu}_B)\iota_B\leq \ker \rt{\mu}$.  Hence, $\coprod\{\ker \rt{\mu}_B: B\in\mathcal{B}\}\leq \ker \rt{\mu}$ and
\begin{align*}
	(\ker \rt{\mu})^{\top} & \leq \left(\coprod\{\ker \rt{\mu}_B: B\in\mathcal{B}\}\right)^{\top}
		 = \prod \{ (\ker \rt{\mu}_B)^{\top} : b\in\mathcal{B}\}\\
		& = \prod \{ \im \lt{\mu}_B : b\in\mathcal{B}\} = \im \lt{\mu}.
\end{align*}
By \propref{prop:nondeg-adj}, to show that $\mor{\mu}$ is nondegenerate it remains to show that $(\ker \lt{\mu})^{\bot}\leq \im \rt{\mu}$.  By \thmref{thm:perp-lattice},
\begin{align*}
	(\ker \lt{\mu})^{\bot} & = \left(\bigcap\{\ker \lt{\mu}_B: B\in\mathcal{B}\}\right)^{\bot}
		 = \left(\sum \{\im \rt{\mu}_B : B\in\mathcal{B}\}\right)^{\top\bot}
		 = (\im \rt{\mu})^{\bot\top}.
\end{align*}
Therefore, it is necessary and sufficient to show $\im \rt{\mu}$ is $\top\bot$-stable.

(iii) follows from (ii) and \corref{coro:comp-rule}.
\end{proof}

\subsection{Division maps}\label{sec:div}
We finish our treatment of nondegenerate bimaps by demonstrating how division maps play the part of ``simple'' nondegenerate bimaps with respect to adjoint-morphisms.  Say that a bimap $B$ is \emph{nondegenerate-simple} if $B$ is nonzero and every
nondegenerate adjoint-morphism from $B$ is either zero or a monomorphism.  

\begin{thm}\label{thm:division}
The nondegenerate-simple bimaps are division maps.
\end{thm}
\begin{proof}
Suppose that $B$ is nondegenerate-simple.  
Let $u\in \lt{B}$ and $v\in \rt{B}$ such that $uBv=0$.  
Define the bimap $C:\lt{B}/u^{\bot\top}\times u^{\perp}\to W$ such that 
\begin{align}
	(x+u^{\bot\top})Cy & =xBy & (\forall x\in \lt{B},\forall y\in u^{\bot}).
\end{align}
Note $C$ is well-defined.
Take $(\pi:\lt{B}\onto \lt{B}/u^{\bot\top},\iota:u^{\perp}\into \rt{B})$ which is an adjoint-epimorphism from $B$ to $C$.  As $B$ is nondegenerate and $\rt{C}^{\top}=u^{\bot\top}/u^{\bot\top}=0$, $C$ is also left-nondegenerate.   Now $(\ker \pi)^{\bot}=u^{\bot\top\bot}=u^{\bot}=\im \iota$.  By \propref{prop:iso-nondeg}(ii), $(\pi,\iota)$ is nondegenerate.  As $B$ is nondegenerate-simple, $(\pi,\iota)$ is zero or monic.  If it is zero then $v\in u^{\bot}=0$ so that $v=0$.  If $(\pi,\iota)$ is monic then $u\in u^{\bot\top}=0$ so that $u=0$.  In every case, if $uBv=0$ then $u=0$ or $v=0$ so that $B$ is a division map.

Next suppose that $B$ is a division map.  Let $\mor{\mu}:B\to C$ be an 
a nondegenerate adjoint-morphism.  As $(\ker \lt{\mu})B(\im \rt{\mu})=0$ and $B$ is a division map, either $\ker \lt{\mu}=0$ or $\im \rt{\mu}=0$.  Now invoke \propref{prop:nondeg-adj}: if $\ker \lt{\mu}=0$ then $\im \rt{\mu}=(\ker \lt{\mu})^{\bot}=\rt{B}$ so that $\lt{\mu}$ is monic and $\rt{\mu}$ is epic, that is, $\mor{\mu}$ is monic; otherwise, $\im \rt{\mu}=0$ so $\ker \lt{\mu}=(\im \rt{\mu})^{\top}=\lt{B}$ and $\mor{\mu}=(0,0)$.  As $\mor{\mu}$ was arbitrary, $B$ is nondegenerate-simple.
\end{proof}
\section{Division algebras up to isotopism}\label{sec:close}

An isotopism $(\phi,\gamma;\kappa)=(\phi \kappa^{-1},\gamma \kappa^{-1};1)(\kappa,\kappa;\kappa)$ is a product of an isomorphism and a principal isotopism.  Also,  $(\phi\kappa^{-1},\kappa\gamma^{-1})$ is an adjoint-isomorphism.  In fact:
\begin{coro}\label{thm:iso-iso}
For a commutative ring $k$, the principal isotopism classes of division $k$-algebras on $W$ is in bijection with the 
class of nondegenerate-simple bimaps $B:W\times W\to W$ up to adjoint-isomorphism.
\end{coro}

\section*{Acknowledgments}

Thanks to the referee for better reasoning, to D. Oury for helpful discussion, and to W. Kantor and D. Shapiro for remarks on semifields and nonsingular bimaps.


\end{document}